\documentclass[a4paper,oneside]{amsart}
\usepackage{amssymb}
\usepackage{amsmath}
\usepackage{mathtools, float}
\usepackage{amsfonts}
\usepackage{stmaryrd}
\usepackage[english]{babel}
\usepackage[utf8]{inputenc}
\usepackage{enumerate}
\usepackage{graphicx,pinlabel,xcolor}

\theoremstyle{plain}
\newtheorem{theorem}{Theorem}[section]
\newtheorem{corollary}[theorem]{Corollary}

\newtheorem{conjecture}{Conjecture}
\newtheorem{theorem*}{Theorem}

\theoremstyle{definition}
\newtheorem{definition}[theorem]{Definition}
\newtheorem{example}[theorem]{Example}

\newtheorem{remark}[theorem]{Remark}
\newtheorem{proposition}[theorem]{Proposition}
\newtheoremstyle{case}{}{}{}{}{}{:}{ }{}
\theoremstyle{case}
\newtheorem{case}{Case}

\newcommand{\wt}{\text{wt}}
\newcommand{\ord}{\text{ord}}
\newcommand{\acs}{\text{acs}}

\newcommand{\codeg}{\text{codeg}}
\newcommand{\diam}{\text{diam}}
\newcommand{\syl}{\text{syl}}

\begin{document}
\title[Bound on the diameter of split metacyclic groups]{Bound on the diameter of \\split metacyclic groups}
\author{Kashyap Rajeevsarathy \and Siddhartha Sarkar}
\address{Department of Mathematics\\
Indian Institute of Science Education and Research Bhopal\\
Bhopal Bypass Road, Bhauri \\
Bhopal 462 066, Madhya Pradesh\\
India}
\email{kashyap@iiserb.ac.in}
\urladdr{https://home.iiserb.ac.in/$_{\widetilde{\phantom{n}}}$kashyap/}

\address{Department of Mathematics\\
Indian Institute of Science Education and Research Bhopal\\
Bhopal Bypass Road, Bhauri \\
Bhopal 462 066, Madhya Pradesh\\
India}
\email{sidhu@iiserb.ac.in}
\urladdr{https://home.iiserb.ac.in/$_{\widetilde{\phantom{n}}}$sidhu/}


\begin{abstract}
Let $G_{m,n,k} = \mathbb{Z}_m \ltimes_k \mathbb{Z}_n$ be the split metacyclic group, where $k$ is a unit modulo $n$. We derive an upper bound for the diameter of $G_{m,n,k}$ using an arithmetic parameter called the \textit{weight}, which depends on $n$, $k$, and the order of $k$. As an application, we show how this would determine a bound on the diameter of an arbitrary metacyclic group. 
\end{abstract}

\maketitle

\section{Introduction}
\label{sec:intro}

The \textit{diameter} of a finite group $G$ with respect to a generating set $S$ is the graph diameter of the Cayley graph $\Gamma(G,S)$ of $G$ with respect to $S$. Consider the semidirect product of the two cyclic groups $\mathbb{Z}_m$ and $\mathbb{Z}_n$ given by the presentation 
 $$G_{m,n,k}:= \mathbb{Z}_m \ltimes_k \mathbb{Z}_n = \langle x,y \,| \,x^m=y^n=1, \, x^{-1} yx = y^k \rangle,$$ where
 $\mathbb{Z}_m = \langle x \rangle$, $\mathbb{Z}_n = \langle y \rangle$, and $k \in {\mathbb Z}^{\times}_n$ of order $\alpha \mid m$, where ${\mathbb Z}^{\times}_n$ denote the group of units of ${\mathbb Z}_n$ with respect to multiplication. We define the diameter of $G_{m,n,k}$ (in symbols $\diam(G_{m,n,k})$) to be the graph diameter of 
 $\Gamma(G_{m,n,k},\{x,x^{-1},y,y^{-1}\}).$  The diameter of finite groups and their bounds have been widely studied, especially from the viewpoint of efficient communication networks (see~\cite{B1,C1} and the references therein). In particular, the networks arising from the Cayley graphs of groups in the subfamily $\{G_{ck,c^2{\ell}, c\ell+1}\}$, also known in computer science parlance as \textit{supertoroids}, have been extensively analyzed~\cite{B1,D4,D3,LV1}. For example, in~\cite{D1,D2}, it was shown that for $c \geq 8$, $\diam(G_{ck,c^2{\ell}, c\ell+1})= [ck/2]+[c\ell/2]$. However, to our knowledge, the diameter bounds for arbitrary groups in $\{G_{m,n,k}\}$ have not been studied. This problem also has connections with the well known degree-diameter problem pertaining to this family of graphs (see~\cite{E1,MH1,VT1}). This is the main motivation behind undertaking such an analysis in this paper. 

Every element of $g \in G_{m,n,k}$ has the unique expression as $g = x^a y^b$. A path $P$ from $1$ to an element $g \in G_{m,n,k}$ would take the form $g = \prod_{i=1}^t x^{a_i} y^{b_i}$. Such a path is said to be \textit{reduced} if $a_i \not\equiv 0 \pmod{m}$, for $2 \leq i \leq t$, and $b_i \not\equiv 0\pmod{n}$, for $1 \leq i \leq t-1$. We define $t$ to be the \textit{syllable} of the reduced path $P$ (as above), and $\sum_{i=1}^{t} |a_i| + |b_i|$ to be its \textit{length} $l(P)$. Denoting by $\mathcal{P}_g$, the collection of all reduced paths in $G$ from $1$ to $g$, we have  
$\|x^a y^b\| = \min \{l(P): P \in \mathcal{P}_g \},$ where $\| \, \|$ is the usual word norm in $G_{m,n,k}$. Thus, the diameter of $G_{m,n,k}$ is given by
\[
\diam(G_{m,n,k}) = {\mathrm {max}} \{ \|x^a y^b\|: 0 \leq a \leq m-1,\, 0 \leq b \leq n-1 \}.
\] 

It is apparent that $[m/2] \leq \diam(G_{m,n,k}) \leq [m/2] + [n/2]$. In reality, $\diam(G_{m,n,k}) = [m/2] + \delta$, where $\delta$ is significantly smaller than $[n/2]$. For example, we can show that $\diam(G_{60,61,2}) = 31$ (see Section~\ref{sec:compute_weights}). In order to obtain a better bound for $\diam(G_{m,n,k})$, we begin by noting that
$$ \prod_{i=1}^t x^{a_i} y^{b_i} = x^{a_1 + \dotsc + a_t} y^{b_1 k^{a_2 + \dotsc + a_t} + \dotsc + b_{t-1} k^{a_t} + b_t}.$$
Consequently, the problem of computing $\| x^a y^b \|$ reduces to the following nonlinear optimization problem in the pair of rings $(\mathbb{Z}_m,\mathbb{Z}_n)$: 
\[ \tag{$\dagger$} \begin{array}{rlr}
\text{minimize} & \displaystyle \sum_{i=1}^{t} |a_i| + |b_i| \, & \text{(in }\mathbb{Z}), \\ 
\text{subject to} & \displaystyle a_1+ \ldots+a_t \equiv a & \pmod{m}, \\ \\
\text{and} & b_1 k^{a_2 + \dotsc + a_t} + \dotsc + b_{t-1} k^{a_t} + b_t \equiv b & \pmod{n}.
\end{array} \]

Fix a positive integer $n \geq 3$, and consider a unit $k \in \mathbb{Z}_n^{\times}$ of multiplicative order $\ord(k) = \alpha \geq 2$. For $0 \leq i \leq \alpha - 1$ and an integer $1 \leq \lambda \leq [n/2]$, a $k$-interval is a set of the form $A(\lambda, i) = \{ a k^i : -\lambda \leq a \leq \lambda \}$. We further reduce the problem of solving $(\dagger)$ to the problem of determining the least positive integer $\lambda = \lambda_0 + \dotsc + \lambda_{\alpha-1}$ so that 
\[
{\mathbb Z}_n = \sum_{i=0}^{\alpha-1} A(\lambda_i, i). 
\] 
We will call this the problem of \textit{covering  the ring ${\mathbb Z}_n$ by sum sets of $k$-intervals}. By showing that the solution to this covering problem in $\mathbb{Z}_n$ depends on two parameters, namely $\wt(n,k;\alpha)$ and $\deg(n,k;\alpha)$ (Section~\ref{sec:comb_tools}), we obtain our main result (Theorem~\ref{thm:main}), which gives a bound for $\diam(G_{m,n,k})$. 

\begin{theorem*}[Main theorem] 
\label{thm:main}
Let $G_{m,n,k}$ be the split metacyclic group given by the presentation
\[
G_{m,n,k} = \langle x, y : x^m = 1 = y^{n}, x^{-1} yx = y^k \rangle,
\]
where $k$ has order $\alpha$ in the group ${\mathbb Z}^{\times}_n$ of units. If $\alpha$ is even and $k^{\alpha/2} \equiv -1 \pmod{n}$, then
\[
\diam(G_{m,n,k}) \leq 
\begin{cases}
[m/2] + \wt(n, k; \alpha), &\mbox{if } \alpha \neq m, \\
[m/2] + \wt(n, k; \alpha) + \deg(n,k;\alpha), &\mbox{if } \alpha = m.
\end{cases}
\]
\end{theorem*}

\noindent Based on our observations, we believe that $\diam(G_{m,n,k}) \leq [m/2] + \wt(n, k; \alpha)$ should hold true, irrespective of the conditions on $m,\,n,\,k$, and $\alpha$. As a direct application of our main result, we obtain an upper bound for the diameter of an arbitrary metacyclic group (Corollary~\ref{cor:diam_bound_for_primes}). 

In practice, it is difficult to compute $\wt(n,k;\alpha)$, or provide a reasonable upper bound for it. Nevertheless, we show that for an odd prime $p$, the growth of $\wt(p^n,k;\alpha)$ is at most linear in $n$ (Corollary {\ref{rem:bound_prime_power}}).

\begin{theorem*}
Let $p$ be an odd prime and $k \in \mathbb{Z}_{p^n}^{\times}$ with $\ord(k) = p^{n-1}(p-1)$. Denote $\wt(p,s;p-1)$ by $\wt(p)$, where $s$ is the image of $k$ under the natural surjection ${\mathbb Z}_{p^n} \rightarrow {\mathbb Z}_p$. Then,
\[
\wt(p) \leq \wt(p^n, k; p^{n-1}(p-1)) \leq 2 n \wt(p). 
\]

\end{theorem*}

\noindent A similar bound is obtained for the case when $p = 2$. Finally, we derive an upper bound of $\wt(p)$, when $p$ is an odd prime (Theorem {\ref{thm:prime_wt_bounds}}).

\begin{theorem*}
 Let $p$ be an odd prime, and let $s \in \mathbb{Z}_p^{\times}$ with $\ord(s) = p-1$. Then, 
\[
\wt(p) \leq 
\begin{cases}
{\frac {p+3}{4}} &\mbox{if } p \equiv 1 \pmod{4}, \\
{\frac {p+5}{4}} &\mbox{if } p \equiv 3 \pmod{4}.
\end{cases}
\]
\end{theorem*}

\section{Some combinatorics pertaining to the covering  problem in ${\mathbb Z}_n$}
\label{sec:comb_tools}

We will now introduce some formal notations to make the problem of covering of ${\mathbb Z}_n$ more precise. Fix a positive integer $n \geq 3$, and consider a unit $k \in \mathbb{Z}_n^{\times}$ of multiplicative order $\ord(k) = \alpha \geq 2$. 

\begin{definition}
\label{defn:omega_set}
Given a pair, ${\underline i} = (i_1, i_2, \dotsc, i_r)$ and ${\underline \lambda} = (\lambda_1, \lambda_2, \dotsc, \lambda_r)$, of sequences of integers such that 
\[
\alpha - 1 \geq i_1 > i_2 > \dotsc > i_r \geq 0 \text{ and } \lambda_j \geq 0, \text{ for } 1 \leq j \leq r,
\] 
we define 
\[
\Omega ({\underline i}, {\underline \lambda}) = \{ b_1 k^{i_1} + \dotsc + b_{r-1} k^{i_{r-1}} + b_r k^{i_r} \pmod{n} \, :\, |b_i| \leq \lambda_i, \, 1 \leq i \leq r\}.
\]
Sometimes we write $\Omega ({\underline i}, {\underline \lambda})$ as $\Omega (k({\underline i}), {\underline \lambda})$, where $k({\underline i}) = (k^{i_1}, \dotsc, k^{i_r})$. We will refer to $i_1$ as the \textit{degree},  and the smallest nonzero number among the $i_k$, for $1 \leq k \leq r$ will be called the \textit{co-degree} of the sequence ${\underline i}$, which we denote by $\deg({\underline i})$ and 
$\codeg({\underline i})$, respectively. The integer $r \geq 1$ will be referred as the \textit{length} of ${\underline i}$.
\end{definition}

 Since $k \in {\mathbb Z}_n^{\times}$, we have $k {\mathbb Z}_n = {\mathbb Z}_n$, and so for each sequence $\underline i$, there exists a finite sequence $\underline \lambda$ (mod $n$) such that $\Omega ({\underline i}, {\underline \lambda}) = {\mathbb Z}_n$. This leads us to the following definition.

\begin{definition}
\label{def:wt}
Given a pair, $\underline{i}$ and $\underline{\lambda}$, of sequences as in Defintion~\ref{defn:omega_set}, we define:
\begin{enumerate}[(i)]
\item The \textit{weight of $\Omega = \Omega ({\underline i}, {\underline \lambda})$} as
$$\wt(\Omega) := \lambda_1 + \lambda_2 + \dotsc + \lambda_r.$$ 
\item The \textit{weight of $(n, k)$ with respect ${\underline i}$} as
$$\wt(n, k; {\underline i}) := \min \{ \wt(\Omega ({\underline i}, {\underline \lambda})) : \Omega ({\underline i}, {\underline \lambda}) = {\mathbb Z}_n \}.$$
\item The \textit{weight of $(n, k)$ of level $r$} as
$$\wt(n, k; r) := {\mathrm {min}} \{ \wt(n, k; {\underline i}) : {\underline i} : \alpha - 1 \geq i_1 > i_2 > \dotsc > i_r \geq 0 \}.$$
\end{enumerate}
\end{definition}

\begin{remark} 
\label{same-cyclic-group}
\noindent From the definition of $\wt(n, k; \alpha)$, it is clear that $\wt(n, k; \alpha) = \wt(n, k^{\prime}; \alpha)$, whenever $k$ and $k^{\prime}$ generate the same cyclic subgroup of ${\mathbb Z}^{\times}_n$. 

\end{remark}

\noindent In our calculations, we will require sequences ${\underline i}$ with $i_r = 0$, which we call \textit{reduced sequences}. 

\begin{definition}
Given a sequence ${\underline i}$ as in Definition~\ref{defn:omega_set}, the \textit{complement  of $\underline{i}$} is the sequence $I(\underline{i}) = (j_1,\ldots,j_r)$ defined by
\[
\alpha - 1 \geq j_1 = \alpha - (i_1 - i_2) > j_2 = \alpha - (i_1 - i_3) > \dotsc > j_{r-1} = \alpha - (i_1 - i_r) > j_r = 0.
\] 
\end{definition}
\noindent In the following proposition, we show that $I(\underline{i})$ is a reduced sequence of length $r$ having the same weight as ${\underline i}$. 

\begin{proposition}
\label{lem:dual_I}
Consider sequences ${\underline i}$ and ${\underline \lambda}$ as in Definition~\ref{defn:omega_set}. If $\Omega ({\underline i}, {\underline \lambda}) = {\mathbb Z}_n$, then
\begin{enumerate}[(i)]
\item $\Omega (I(\underline{i}), {\underline \lambda^{\prime}}) = {\mathbb Z}_n,$ where ${\underline \lambda^{\prime}} = (\lambda_2, \lambda_3, \dotsc, \lambda_r, \lambda_1)$, and
\item $\wt(n, k; {\underline i}) = \wt(n, k; I (\underline{i}))$. 
\end{enumerate}
\end{proposition}

\begin{proof} Given any $b$ (mod $n$), there exists $b_1, \dotsc, b_r$ such that $|b_s| \leq \lambda_s$, for $1 \leq s \leq r$, and $b = b_1 k^{i_1} + b_2 k^{i_2} + \dotsc + b_{r-1} k^{i_{r-1}} + b_r k^{i_r}$. So, we have
\[
b k^{\alpha - i_1} = b_1 + b_2 k^{\alpha - (i_1 - i_2)} + \dotsc + b_r k^{\alpha - (i_1 - i_r)} \in \Omega (I({\underline i}), {\underline \lambda^{\prime}}).
\]
Hence, $k^{\alpha - i_1} \Omega ({\underline i}, {\underline \lambda}) \subseteq \Omega (I({\underline i}), {\underline \lambda^{\prime}})$, and as $k$ is a unit, we have $k^{\alpha - i_1} \Omega ({\underline i}, {\underline \lambda}) = {\mathbb Z}_n$, which establishes (i).

For (ii), note that if $\wt(n, k; {\underline i}) = \lambda$, then $\Omega ({\underline i}, {\underline \lambda}) = {\mathbb Z}_n$, for some sequence ${\underline \lambda} = (\lambda_1, \lambda_2, \dotsc, \lambda_r)$ with $\lambda = \lambda_1 + \lambda_2 + \dotsc + \lambda_r$ being the least possible value. As seen above, we have $\Omega (I({\underline i}), {\underline \lambda^{\prime}}) = {\mathbb Z}_n$, and furthermore ${\underline \lambda^{\prime}} = (\lambda_2, \lambda_3, \dotsc, \lambda_r, \lambda_1)$ yields the same weight $\lambda$. Thus, we have $\wt(n, k; I({\underline i})) \leq \wt(n, k; {\underline i})$. 

Suppose that $\mu = \wt(n, k; I({\underline i})) < \wt(n, k; {\underline i})$. Then there exists a sequence ${\underline \mu} = (\mu_2, \dotsc, \mu_r, \mu_1)$ such that $\Omega (I({\underline i}), {\underline \mu}) = {\mathbb Z}_n$. Multiplying by $k^{i_1}$, we get $\Omega ({\underline i}, {\underline \mu^{\prime}}) = {\mathbb Z}_n$, where ${\underline \mu^{\prime}} = (\mu_1, \mu_2, \dotsc, \mu_r)$. As this contradicts the minimality of $\wt(n, k; {\underline i}) = \lambda$, (iii) follows. 
\end{proof}

\noindent Lemma~\ref{lem:dual_I} implies that it suffices to consider only reduced sequences while computing $\wt(n, k; r)$.

\begin{remark}
For any length $r \leq \alpha$, we can see that $\wt(n, k; r) \leq [n/2]$, by considering the sequence $\underline{\lambda} = (\lambda_1,\ldots,\lambda_r)$, where $\lambda_i = [n/2]$ for a fixed $i$, and $\lambda_j = 0$, for the indices $j \neq i$.
\end{remark}

\begin{definition}
 Given a sequence ${\underline i}: \alpha - 1 \geq i_1 > i_2 > \dotsc > i_r \geq 0$, a sequence ${\underline j} = \alpha - 1 \geq j_1 > j_2 > \dotsc > j_q \geq 0$ of length $q \geq r$ is said to be \textit{finer than ${\underline i}$} (in symbols ${\underline i} \preceq {\underline j}$), if it is obtained from $\underline{i}$ by adding one or more terms. 
\end{definition}

\begin{remark}
\label{rem:poset}
Let ${\mathcal P}(\alpha - 1)$ denote the collection of all sequences of length $\leq \alpha$ as in Definition~\ref{defn:omega_set}, and let ${\mathcal P}^{\prime}(\alpha - 1)$ be the subcollection of all reduced sequences. Note that $\preceq$ defines a partial order on ${\mathcal P}(\alpha - 1)$ under which ${\mathcal P}^{\prime}(\alpha - 1)$ is a subposet of ${\mathcal P}(\alpha - 1)$. 
\end{remark}

\begin{proposition} 
Consider the posets $({\mathcal P}(\alpha - 1), \preceq)$ and $({\mathcal P}^{\prime}(\alpha - 1), \preceq)$ as in Remark~\ref{rem:poset}. Then:
\begin{enumerate}[(i)]
\item For any two elements ${\underline i}, {\underline j} \in {\mathcal P}(\alpha - 1)$ (resp. ${\mathcal P}^{\prime}(\alpha - 1)$), there exists ${\underline \xi} \in {\mathcal P}(\alpha - 1)$ (resp. ${\mathcal P}^{\prime}(\alpha - 1)$) such that ${\underline i} \preceq {\underline \xi}$ and ${\underline j} \preceq {\underline \xi}$. 
\item $({\mathcal P}^{\prime}(\alpha - 1), \preceq)$ has a maximal element $\Delta$, and a minimal element $\delta$ given by 
\begin{eqnarray*}
\Delta & = & (\alpha - 1, \alpha - 2, \dotsc, 1, 0), \text{ and} \\
 \delta & = & (\alpha - 1, 0) 
\end{eqnarray*}
of lengths are $\alpha$ and $2$, respectively. 
\item The map $\Psi : ({\mathcal P}(\alpha - 1), \preceq) \rightarrow ({\mathbb N}, \leq)$ defined by 
\[
\Psi({\underline i}) := wt(n, k; {\underline i})
\]
is an order reversing function, where $({\mathbb N}, \leq)$ is regarded as a linearly ordered poset with respect to natural order.
\end{enumerate}
\end{proposition}

\begin{proof}
Given sequences ${\underline i}, {\underline j} \in {\mathcal P}(\alpha - 1)$ (resp. ${\mathcal P}^{\prime}(\alpha - 1)$), consider the sequence ${\underline \xi}$ obtained by taking the union of elements in ${\underline i}$ and ${\underline j}$,  rearranged in decreasing order. Then, clearly ${\underline i} \preceq {\underline \xi}$ and ${\underline j} \preceq {\underline \xi}$, from which (i) and (ii) follow.

For showing (iii), consider sequences ${\underline i} \preceq {\underline j}$ with lengths $r < s$ such that $\wt(n, k; {\underline i}) = \lambda$ is realized by a sequence ${\underline \lambda} = (\lambda_1, \lambda_2, \dotsc, \lambda_r)$. Define ${\underline \mu} = (\mu_1, \dotsc, \mu_s)$ by $\mu_t = \lambda_t$, if $j_t$ is an element in ${\underline i}$, and $\mu_t = 0$ otherwise. Then $\lambda = \mu_1 + \dotsc + \mu_s$ and $\Omega ({\underline j}, {\underline \mu}) = {\mathbb Z}_n$. Hence, we have that $\wt(n, k; {\underline j}) \leq \lambda = \wt(n, k; {\underline i})$, and (iii) follows. 
\end{proof} 

\begin{remark}
\label{rem:wt_and_Delta}
Clearly, $\wt(n, k; \delta) = [n/2]$, and the only sequence of length $\alpha$ is the maximal element $\Delta$. Thus, $\wt(n, k; \Delta) = \wt(n, k; \alpha)$, which shows the importance of analyzing $\wt(n, k; \Delta)$.
\end{remark}

\begin{definition}
\label{min-prime}
 A sequence ${\underline i} \in {\mathcal P}^{\prime}(\alpha -1)$ is called a \textit{minimal prime sequence realizing $\wt(n,k;\alpha)$} if:
\begin{enumerate}[(i)]
\item ${\underline i} \in S(\alpha) = \{ {\underline j} \in {\mathcal P}^{\prime}(\alpha -1) : \wt(n, k; {\underline j}) = \wt(n, k; \alpha) \}$, and 
\item ${\mathrm{length}}({\underline i}) = \min \{ {\mathrm{length}}({\underline j}) : {\underline j} \in S(\alpha) \}$  
\end{enumerate}
We denote the smallest possible degree of a minimal prime sequence realizing $\wt(n,k;\alpha)$ by $\deg(n,k;\alpha)$.
\end{definition}

\begin{example}
 When $(n, k) = (30, 7)$, $\ord(7) = 4$ in ${\mathbb Z}^{\times}_{30} \cong \mathbb{Z}_2 \times \mathbb{Z}_4$. We consider the sequences $\underline{i_1}: 1,0$, $\underline{i_2}: 2,0$, and $\underline{i_3}: 3,0$. For each sequence $\underline{i_k}$, in Table~\ref{tab:wts_omega} below, we list some possible choices of a sequence $\underline{\lambda} : \lambda_1,\lambda_2$ (as in Definition~\ref{defn:omega_set}), and the values of $\wt(\Omega(\underline{i},\underline{\lambda})).$
 
\begin{table}[h]
\begin{center}
  \begin{tabular}{| c | l | c | l | }
    \hline
    $\underline{i}$ & $\lambda_1$ &  $\lambda_2$ & $\wt(\Omega(\underline{i}, \underline{\lambda}))$ \\ \hline
    $\underline{i_1}$ & 0 & 15 & 15 \\ \hline
    $\underline{i_1}$ & 1 & 8 & 9 \\ \hline
    $\underline{i_1}$ & 2 & 3 & 5 \\ \hline
    $\underline{i_1}$ & 3 & 3 & 6 \\ \hline
    $\underline{i_1}$ & 4, 5 & 2 & 6, 7 \\ \hline
    $\underline{i_1}$ & 6, $\dotsc$, 14 & 1 & 7, $\dotsc$, 15 \\ \hline
    $\underline{i_1}$ & 15 & 0 & 15 \\ \hline
    $\underline{i_2}$ & 0 & 15 & 15 \\ \hline
    $\underline{i_2}$ & 1 & 5 & 6 \\ \hline
    $\underline{i_2}$ & 2, 3 & 4 & 6, 7 \\ \hline
    $\underline{i_2}$ & 4 & 2 & 6 \\ \hline
    $\underline{i_2}$ & 5, $\dotsc$, 14 & 1 & 6, $\dotsc$, 15 \\ \hline
    $\underline{i_2}$ & 15 & 0 & 15 \\ \hline
    $\underline{i_3}$ & 0 & 15 & 15 \\ \hline
    $\underline{i_3}$ & 1 & 6 & 7 \\ \hline
    $\underline{i_3}$ & 2 & 4 & 6 \\ \hline
    $\underline{i_3}$ & 3, $\dotsc$, 7 & 2 & 5, $\dotsc$, 9 \\ \hline
    $\underline{i_3}$ & 8, $\dotsc$, 14 & 1 & 9, $\dotsc$, 15 \\ \hline
    $\underline{i_3}$ & 15 & 0 & 15 \\ \hline
  \end{tabular}
  \caption{Some computations of $\wt(\Omega(\underline{i}, \underline{\lambda}))$.}
   \label{tab:wts_omega}
\end{center}
\end{table}
\noindent A direct calculation (using software written for Mathematica 11~\cite{W1}) shows that $\wt(30, 7; 4) = 5$. Hence, $\underline{i_1}$  and $\underline{i_3}$ are minimal prime sequences that realize this weight, and $S(4) = \{ (3,0), (1,0) \}$.
\end{example} 

\noindent This example shows that in practice it is difficult to compute $\wt(n,k;\alpha)$ in most situations. However, we will now obtain reasonable bounds on $\wt(n,k;\alpha)$ in case $n$ is a prime power. 

\section{Bounds on the weight of a prime power}
It is well known that when $p$ is an odd prime, for $n \geq 1$, we have ${\mathbb Z}^{\times}_{p^n} \cong \mathbb{Z}_{p^{n-1}(p-1)}$, and for $n \geq 2$, we have ${\mathbb Z}^{\times}_{2^n} \cong \mathbb{Z}_2 \times \mathbb{Z}_{2^{n-2}}$. For $n \geq 2$, let $\varphi_n$ denote the natural quotient ring homomorphism ${\mathbb Z}_{p^n} \rightarrow {\mathbb Z}_{p^{n-1}}$. If $k \in {\mathbb Z}^{\times}_{p^n}$ has order $p^{n-1}(p-1)$, then $\varphi_n(k)$ generates the cyclic group ${\mathbb Z}^{\times}_{p^{n-1}}$. Denoting the epimorphism $\varphi_n \vert_{{\mathbb Z}^{\times}_{p^n}}$ by 
$\widetilde{\varphi}_n$, for an arbitrary unit $k \in {\mathbb Z}^{\times}_{p^n}$, we have
\[
{\mathrm {ord}}(\widetilde{\varphi}_n(k)) = 
\begin{cases}
{\mathrm {ord}}(k)/p, &\mbox{if } p \mid {\mathrm {ord}}(k), \text{ and}\\
{\mathrm {ord}}(k), &\mbox{otherwise.} 
\end{cases}
\]
\noindent When $n \geq 1 $, we derive bounds for $wt(p^n, k; \ord(k))$ in terms of the  $wt(p^n, \widetilde{\varphi}_n(k); \newline \ord(\widetilde{\varphi}_n(k)))$. We first consider the case when $n=1$.

\begin{theorem} 
\label{thm:prime_wt_bounds}
Let $p$ be an odd prime, and let $s \in \mathbb{Z}_p^{\times}$ with $\ord(s) = p-1$. 
\begin{enumerate}[(i)]
\item If $p \equiv 1 \pmod{4}$, then $wt(p, s; p-1) \leq \frac {p+3}{4}$. Moreover, there exists a sequence $\underline{i}$ such that $\deg(\underline{i}) = \frac {p-5}{4}$ and $\wt(p,s;\underline{i}) \leq \frac {p+3}{4}$.

\item If $p \equiv 3 \pmod{4}$, then $wt(p, s; p-1) \leq {\frac {p+5}{4}}$. Moreover, there exists a sequence $\underline{i}$ such that $\deg(\underline{i}) = \frac {p-3}{4} $ and $\wt(p,s;\underline{i}) \leq \frac {p+5}{4}$.
\end{enumerate}
\end{theorem}

\begin{proof} We present a proof only for (i), as (ii) will follow from similar arguments. Consider the list of units $A = \{ s^{i_1}, s^{i_1 - 1}, \dotsc, s, 1 \}$, where $i_1 = {\frac {p-1}{4}} - 1$. Since $s^{\frac {p-1}{2}} \equiv -1 \pmod{p}$, the set 
\[
A \cup -A = \{ s^j ~:~ 0 \leq j \leq {\frac {p-1}{4}} - 1 ~{\mathrm {or}}~ {\frac {p-1}{2}} \leq j \leq {\frac {p-1}{2}} + i_1 \}
\]
comprises exactly half of the elements of ${\mathbb Z}^{\times}_p$, that is, $|A \cup -A| = {\frac {p-1}{2}}$. Now let $s^{\tau} \not\in A \cup -A$. For every $s^j \in A \cup -A$, there are ${\frac {p-1}{2}}$ distinct elements among the $s^{\tau} - s^j$, none of which is equal to $s^{\tau}$. Since there are only ${\frac {p-1}{2}} - 1$ elements outside $A \cup -A$ other than $s^{\tau}$, one of these elements must be from $A \cup -A$. Hence, there exists $s^{j_1} \in A \cup -A$ such that $s^{\tau} = s^j + s^{j_1}$. 

Suppose that $j = j_1$. Then write $s^{\tau} = s^j - (-s^j) = s^j - s^{j+ {\frac {p-1}{2}}}$. If $s^j \in A$, then $s^{j+ {\frac {p-1}{2}}} \in -A$. Otherwise, if $s^j \in -A$, then we have $s^{j+ {\frac {p-1}{2}}} \in A$, except when $j = {\frac {p-1}{2}}$. In this particular case, we have $s^{\tau} = -2$, which implies that $\Omega({\underline i}, {\underline \lambda}) = {\mathbb Z}_p$ where 
\[
{\underline i} = (i_1, i_1 - 1, \dotsc, 1, 0) \text{ and } \underline{\lambda} = (\underbrace{1, \dotsc, 1}_{i_1}, 2). 
\]
\end{proof}

\begin{remark}
Note that the only difficulty in the proof of Theorem~\ref{thm:prime_wt_bounds} was to represent $s^{\tau} = -2 \not\in A \cup -A$. But this situation does not arise when $2 \in A \cup -A$, in which case we have ${\underline \lambda} = (\underbrace{1, \ldots, 1}_{i_1 + 1})$, and we obtain the following slightly improved bound
$$\wt(p,s;p-1) \leq \begin{cases}
                                \frac{p-1}{4}, & \text{if } p \equiv 1 \pmod{4}, \text{ and} \\
                                \frac{p+1}{4}, & \text{if } p \equiv 3 \pmod{4}.
                                \end{cases}$$                                              
\end{remark}

\begin{remark}
One might also consider applying a generalization of the Cauchy-Davenport theorem (\cite[Theorem 2.3]{NAT}) in the proof of Theorem~\ref{thm:prime_wt_bounds}. However, this would yield the bound $\wt(p,s;p-1) \leq {\frac {p-1}{2}}$, which is significantly weaker than the one we have derived.
\end{remark}

\begin{definition}
For a prime $p$ and $s \in {\mathbb Z}^{\ast}_p$ with ${\mathrm{ord}}(s) = p-1$ we define 
\[
\wt(p) := \wt(p,s;p-1)
\]
\end{definition}

\noindent From the discussion after definition {\ref{def:wt}} it follows that $\wt(p)$ does not depend on the choice of $s$. 

\begin{theorem}[An upper bound]
\label{thm:wt_prime_power}
For a fixed prime $p$ and $n > 1$, consider $k \in {\mathbb Z}^{\times}_{p^n}$ with $\ord(k) = m$, and $k_0 = \widetilde{\varphi}_n(k) \in {\mathbb Z}^{\times}_{p^{n-1}}$ with $\ord(k_0) = m_0$. Then, 
\[
\wt(p^n, k; m) \leq 
\begin{cases}
\wt(p^{n-1}, k_0; m_0) + 2 \wt(p),  & \mbox{if } p \mid m,  \text{ and}\\
p \wt(p^{n-1}, k_0; m_0) + [p/2], &\mbox{if } p ~{\mathrm{odd}} \text{ and } p \nmid m.
\end{cases} 
\]
\end{theorem}

\begin{proof} First note that while $p = 2$, we must have $2 \mid m$ and hence these are the only cases. Let $\lambda^{\prime} = \wt(p^{n-1}, \widetilde{\varphi}_n(k); m_0)$ be realized by a sequence $\lambda^{\prime}_1,  \dotsc, \lambda^{\prime}_{m_0}$ such that every $a^{\prime} \in {\mathbb Z}_{p^{n-1}}$ is expressed as 
\[
a^{\prime} \equiv b^{\prime}_1 \widetilde{\varphi}_n(k)^{m_0 - 1} + b^{\prime}_2 \widetilde{\varphi}_n(k)^{m_0 - 2} + \dotsc + b^{\prime}_{m_0 - 1} \widetilde{\varphi}_n(k) + b^{\prime}_{m_0} \pmod{p^{n-1}}
\]
for integers $b^{\prime}_j$ with $|b^{\prime}_j| \leq \lambda^{\prime}_j$. Then every $a \in {\mathbb Z}_{p^n}$ can be expressed as
$$a \equiv b^{\prime}_1 k^{m_0 - 1} + b^{\prime}_2 k^{m_0 - 2} + \dotsc + b^{\prime}_{m_0 - 1} k + b^{\prime}_{m_0} + z_a \pmod{p^n},$$

with $|b^{\prime}_j| \leq \lambda^{\prime}_j$, and $z_a \in {\mathrm {ker}}(\varphi_n)$, which depends on $a$. Note that
\[
{\mathrm {ker}}(\varphi_n) = \{ \xi p^{n-1} : 0 \leq \xi \leq p-1 \} \text{ and } {\mathrm {ker}}(\widetilde{\varphi}_n) = \langle k^{m_0} \rangle \leq \langle k \rangle = {\mathbb Z}^{\times}_{p^n}.
\]

 If $p \mid m$, then we have $m_0 = m/p$, and so
\[
\mathrm {ker}({\varphi}_n) = \{k^{\tau m_0}-1 : 0 \leq \tau \leq p-1\}. 
\]

So, we have a bijection $\text{ker}(\varphi_n)$ onto itself given by
\[
\theta p^{n-1} \mapsto k^{i_{\theta}m_0} - 1 \text{ whose inverse is }  k^{\tau m_0} - 1 \mapsto \theta_{\tau} p^{n-1},
\]
satisfying $i_0 = 0$ and $\theta_0 = 0$.  Now following the notations in Theorem {\ref{thm:prime_wt_bounds}}, fix $s \in {\mathbb Z}^{\times}_p$ of order $p-1$, and let $\wt(p) = \wt(\Omega(\underline{i}, \underline{\mu}))= \mu$, where ${\underline \mu} = (\mu_1, \dotsc, \mu_p)$ and $\underline{i} = (p-1, \dotsc, 1, 0)$. Using Theorem {\ref{thm:prime_wt_bounds}}, any $\theta p^{n-1} \in {\mathrm{ker}}(\varphi_n)$ is represented as 
\begin{eqnarray*}
\theta p^{n-1} & \equiv & c_1 s^{p-1} p^{n-1} + \dotsc + c_{p-1} s p^{n-1} + c_p p^{n-1} \pmod{p^n} \\
	& \equiv & c_1 \bigg( k^{i_{s^{p-1}} m_0} - 1 \bigg) + \dotsc + c_{p-1} \bigg( k^{i_s m_0} - 1 \bigg) + c_p \bigg( k^{i_1 m_0} - 1 \bigg) \pmod{p^n} \\
	& \equiv & c_1 k^{i_{s^{p-1}} m_0} + \dotsc + c_{p-1} k^{i_s m_0} + c_p k^{i_1 m_0} - (c_1 + \dotsc + c_p) k^{i_0 m_0} \pmod{p^n},
\end{eqnarray*}
where $|c_i| \leq \mu_i$. By considering an appropriate rearrangement of the following sequences
\begin{eqnarray*}
k(\underline{i}') & = & (\underbrace{k^{i_{s^{p-1}} m_0}, \dotsc, k^{i_s m_0}, k^{i_1 m_0}}_p, k^{m_0 - 1}, k^{m_0 - 2}, \dotsc, k, 1)  \text{ and} \\
\underline{\lambda}' & = & (\mu_1, \dotsc, \mu_p, \lambda^{\prime}_1, \dotsc, \lambda^{\prime}_{m_0 - 1}, \lambda^{\prime}_{m_0} + \mu_1 + \dotsc + \mu_p),
\end{eqnarray*} 
where $\underline{i}'$ is the sequence of the exponents, we obtain sequences, ${\underline i}$ an ${\underline \lambda}$, respectively, such that ${\mathbb Z}_{p^n} = \Omega({\underline i}, {\underline \lambda})$. This proves the first inequality.

 If $p \nmid m$, then $m = m_0$, and so $p$ must be odd. From the discussion above, there exists $\xi \in {\mathbb Z}_p$ so that
$${\frac {p^{n-1} - 1}{2}} \equiv b^{\prime\prime}_1 k^{m_0 - 1} + b^{\prime\prime}_2 k^{m_0 - 2} + \dotsc + b^{\prime\prime}_{m_0 - 1} k + b^{\prime\prime}_{m_0} + \xi p^{n-1} \pmod{p^n},$$
which implies that
$$(-2\xi + 1) p^{n-1} \equiv 2b^{\prime\prime}_1 k^{m_0 - 1} + 2b^{\prime\prime}_2 k^{m_0 - 2} + \dotsc + 2b^{\prime\prime}_{m_0 - 1} k + (2b^{\prime\prime}_{m_0} + 1) \pmod{p^n}$$
and that $\xi_0 p^{n-1} := (-2\xi + 1) p^{n-1}$ is a non-trivial element of ${\mathrm {ker}}(\varphi_n)$. Hence, we have that ${\mathrm {ker}}(\varphi_n) = \{ \tau \xi_0 p^{n-1} : -[p/2] \leq \tau \leq [p/2] \}$. Now setting 
\begin{eqnarray*}
{\underline i} & = & (m_0 - 1, \dotsc, 1, 0) \text{ and} \\
{\underline \lambda^{\prime}} & = & \bigg( (2[p/2]+1)\lambda^{\prime}_1, \dotsc, (2[p/2]+1)\lambda^{\prime}_{m_0 - 1}, (2[p/2]+1)\lambda^{\prime}_{m_0} + [p/2] \bigg),  
\end{eqnarray*} 
we see that ${\mathbb Z}_{p^n} = \Omega({\underline i}, {\underline \lambda^{\prime}})$, which proves the second inequality.
\end{proof}

\noindent Using arguments similar to the ones used in the second half of the proof of Theorem {\ref{thm:wt_prime_power}}, we obtain the following corollary.

\begin{corollary}
For a fixed odd prime $p$ and $n > 1$, consider $k \in {\mathbb Z}^{\times}_{p^n}$ with $\ord(k) = m$, and $k_0 = \widetilde{\varphi}_n(k) \in {\mathbb Z}^{\times}_{p^{n-1}}$ with $\ord(k_0) = m_0$. Assuming that $p \nmid m$, for each $\xi p^{n-1} \in {\mathrm{ker}}({\varphi}_n)$, let $\mu_{\xi}$ denote the minimum of all sums $\mu_1 + \dotsc + \mu_{m_0}$ so that 
\[
\xi p^{n-1} \equiv b^{\prime}_1 k^{m_0 - 1} + b^{\prime}_2 k^{m_0 - 2} + \dotsc + b^{\prime}_{m_0 - 1} k + b^{\prime}_{m_0} \pmod{p^n},
\]
where $|b^{\prime}_j| \leq \mu_j$. Then 
\[
\wt(p^n, k; m) \leq 
[p/2](2\mu_0 + 1) + \wt(p^{n-1}, k_0; \mu_0), 
\]
where $\mu_0 = \min \{ \mu_{\xi} : 1 \leq \xi \leq p-1 \}$.
\end{corollary}

\noindent The following is a direct consequence of Theorem~\ref{thm:prime_wt_bounds}.
\begin{corollary}
\label{rem:bound_prime_power}
Let $p$ be a prime.
\begin{enumerate}[(i)]
\item If $p$ is odd and $k \in \mathbb{Z}_{p^n}^{\times}$ with $\ord(k) = p^{n-1}(p-1)$, then
\[
\wt(p) \leq \wt(p^n, k; p^{n-1}(p-1)) \leq 2n \wt(p), 
\]
where the $\epsilon(p) = +1$, if $p \equiv 1 \pmod{4}$, and is $-1$, otherwise.
\item If $p= 2$, $n \geq 4$, and $k \in \mathbb{Z}_{2^n}^{\times}$ with $\ord(k) = 2^{n-2}$, then 
\[
\gamma(k) \leq \wt(2^n, k; 2^{n-2}) \leq \gamma(k) + 2(n-1), 
\]
where $\gamma(k) = 4$, if $k^{2^{n-3}} \equiv -1 \pmod{2^n}$, and is $2$, otherwise. 
\end{enumerate}
\end{corollary}

Based on our observations, we believe that the following conjectures have to hold true. However, this will require much deeper combinatorics to establish them. 

\begin{conjecture} Let $p$ be an odd prime. Then there exists a constant $C(p) \approx 1$ such that
\[
\wt(p) \leq C(p) \log_2(p).
\]

\end{conjecture}

\begin{conjecture} Let $p$ be an odd prime and $n > 1$. Let $\varphi_{n} : {\mathbb Z}_{p^n} \rightarrow {\mathbb Z}_{p^{n-1}}$ denote the natural surjective morphism. Suppose $k \in {\mathbb Z}^{\times}_{p^n}$ with $\ord(k) = pm_0$, and $k_0 = \varphi_{n}(k) \in {\mathbb Z}^{\times}_{p^{n-1}}$ with $\ord(k_0) = m_0$. Then, 
\[
\wt(p^{n-1}, k_0; m_0) + \wt(p) - 1 \leq \wt(p^n, k; pm_0) \leq \wt(p^{n-1}, k_0; m_0) + \wt(p). 
\] 

\end{conjecture}

For a sequence $\underline{i}$ realizing the weight corresponding to a unit of maximum order in ${\mathbb Z}^{\times}_{p^n}$, we have:
\[
\deg(\underline{i}) \leq 
\begin{cases}
{\frac {p-1}{4}} + (n-1)p - 1, &\mbox{if } p \equiv 1 \pmod{4}, \text{ and}\\
{\frac {p-3}{4}} + (n-1)p, &\mbox{if } p \equiv 3 \pmod{4}.
\end{cases}
\]

\noindent The final result in this section gives a bound on the degree of a minimal prime sequence realizing $\wt(n,k;\alpha)$, when $\alpha$ is even and $k^{\alpha/2} \equiv -1 \pmod{n}$.

\begin{proposition} 
\label{thm:bound_for_deg_nkalpha}
Let $k \in {\mathbb Z}^{\times}_n$ with $\ord(k) = \alpha > 1$, where $\alpha$ is even and $k^{\alpha/2} \equiv -1 \pmod{n}$. Then $\deg(n,k;\alpha) \leq \alpha/2$.  
\end{proposition}

\begin{proof} We know from Remark~\ref{rem:wt_and_Delta} that $\wt(n,k;\Delta) = \wt(n,k;\alpha)$, where $\Delta:\alpha-1, \dotsc, 1, 0$. For this $\Delta$, let $wt(n,k;\alpha)$ be realized by a sequence $\underline{\lambda}:\lambda_1, \ldots, \lambda_{\alpha}$, so that each element of $\Omega(\underline{i},\underline{\lambda})$ has a representation of the form
\[
b_1 k^{\alpha - 1} + b_2 k^{\alpha - 2} + \dotsc + b_{\alpha - 1} k + b_{\alpha}, 
\]
where $|b_i| \leq \lambda_i$. 
Replacing the powers $k^j$, for $\alpha - 1 \geq j \geq \alpha/2$, in this representation by $-k^{j - \alpha/2}$, we obtain an expression of the form
\[
(b_{\alpha/2 +1} - b_1) k^{\alpha/2 - 1} + (b_{\alpha/2 +2} - b_2) k^{\alpha/2 - 2} + \dotsc + (b_{\alpha - 1} - b_{\alpha/2 - 1}) k + (b_{\alpha} - b_{\alpha/2}),
\]
which represents an element of $\Omega(\underline{i}',\underline{\lambda}')$, where $\underline{i}': \alpha/2-1,\ldots,1,0$ and $\underline{\lambda}' :\lambda_{\alpha/2 +1} + \lambda_1, \dotsc, \lambda_{\alpha} + \lambda_{\alpha/2}$. The result now follows from the definition of a minimal prime sequence realizing $\wt(n,k;\alpha)$.  
\end{proof}
 
\section{Bounding the diameter of split metacyclic groups}
 There are two key steps involved in solving our main optimization problem ($\dagger$). In the first step (first reduction), we restrict our optimization to the component ring $\mathbb{Z}_n.$ In the second step (second reduction), we build on the results of the first step towards arriving at a solution to $(\dagger)$. We fix the notation that $k \in {\mathbb Z}^{\times}_n$ with $\ord(k) = \alpha$, and regard the elements of ${\mathbb Z}_n$ as formal sums
\[
w({\underline b}, {\underline c}) = b_1 k^{c_1} + b_2 k^{c_2} + \dotsc + b_t k^{c_t},
\]
where ${\underline b} = (b_1, b_2, \dotsc, b_t)$ and ${\underline c} = (c_1, c_2, \dotsc, c_t)$ are two arbitrary integer sequences. Further, we will abuse notation by using the same expression of $w({\underline b}, {\underline c})$ while treating the sum as an element of ${\mathbb Z}_n$. The first reduction step (Proposition {\ref{prop:first_red_step}}) will connect these to the sequences which we have introduced in Definition {\ref{defn:omega_set}}. 
\begin{definition}
The formal sum $w({\underline b}, {\underline c})$ is called \textit{primal}, if for any $c$ (mod $\alpha$), there is at the most one non-zero entry among $b_t$ with $c_t \equiv c$ (mod $\alpha$). We call the number $|b_1| + |b_2| + \dotsc + |b_t|$ as the \textit{absolute coefficient sum} of $w({\underline b}, {\underline c})$, which we denote by $\acs(w({\underline b}, {\underline c}))$. The \textit{ordered sequence relative to} $w({\underline b}, {\underline c})$ is defined to be $S(w({\underline b}, {\underline c})) = (i_1, \dotsc, i_r)$ such that 
\[
\alpha - 1 \geq i_1 > \dotsc > i_r \geq 0 \hspace*{.2in} {\mathrm{and}} \hspace*{.2in} \{ k^{c_1}, \dotsc, k^{c_t} \} = \{ k^{i_1}, \dotsc, k^{i_r} \}.
\] 
\end{definition}
\noindent  For example, the sum $1. k^2 + 0. k^2 + 1. k + 1$ is primal, while $2. k^2 - 1. k^2 + 1. k + 1$ is not, and the ordered sequence related to both of them is $(2, 1, 0)$. Our first step of reduction process involves reducing the absolute coefficient sum of $w({\underline b}, {\underline c})$ without changing the value of $w({\underline b}, {\underline c})$ (mod $n$), and the powers $k^{c_1}, k^{c_2}, \dotsc, k^{c_t}$, while retaining their multiplicities. For example we want to reduce $2. k^2 - 1. k^2 + 1. k + 1$ to $1. k^2 + 0. k^2 + 1. k + 1$. Keeping the same powers of $k$ with zero coefficient in the reduction leads to the idea of connecting the main optimization problem to the problem of covering finite rings by powers of the same unit. 

\begin{proposition}[First reduction step] 
\label{prop:first_red_step}
Consider the formal sum 
\[
w({\underline b}, {\underline c}) = b_1 k^{c_1} + b_2 k^{c_2} + \dotsc + b_t k^{c_t},
\]
and let ${\underline i} = (i_1, \dotsc, i_r)$ be the (ordered) sequence with $\alpha - 1 \geq i_1 > \dotsc > i_r \geq 0$ and 
\[
\{ k^{c_1}, \dotsc, k^{c_t} \} = \{ k^{i_1}, \dotsc, k^{i_r} \}
\]
Then there exists a sequence ${\underline b^{\prime}} = (b^{\prime}_1, b^{\prime}_2, \dotsc, b^{\prime}_t)$ of integers such that 
\begin{enumerate}[(i)]
\item $w({\underline b^{\prime}}, {\underline c})$ is primal, as a formal element,
\item $w({\underline b^{\prime}}, {\underline c}) \equiv w({\underline b}, {\underline c}) \pmod{n}$,
\item $\acs(w({\underline b^{\prime}}, {\underline c})) \leq$ $\acs(w({\underline b}, {\underline c}))$, and 
\item $\acs(w({\underline b^{\prime}}, {\underline c})) \leq \wt(n, k; {\underline i})$.
\end{enumerate}
\end{proposition}
\begin{proof} 
We first write 
\[
w({\underline b}, {\underline c}) = b_1 k^{c_1} + b_2 k^{c_2} + \dotsc + b_t k^{c_t} = s_1 k^{i_1} + s_2 k^{i_2} + \dotsc + s_r k^{i_r},
\]
where $\alpha - 1 \geq i_1 > i_2 > \dotsc > i_r \geq 0$ with $r$ being the number of distinct powers of $k$ in the formal sum on the left, and $s_j = \sum_{c_q \equiv i_j \pmod{\alpha}} b_q$.  Then any sequence ${\underline b^{\prime}}:  b^{\prime}_1, b^{\prime}_2, \dotsc, b^{\prime}_t$ obtained by replacing exactly one of the elements in each collection $\{b_q : c_q \equiv i_j \pmod{\alpha}\}$ by $s_j$, and the remaining by $0$ satisfies conditions 
(i) - (ii). We obtain (iii) by applying the triangle inequality to the expression for $s_j$. Finally, if $\acs(w({\underline b^{\prime}}, {\underline c})) > \wt(n, k; {\underline i})$, then by definition of $wt(n, k; {\underline i})$, we may replace the sequence $s_1, \dotsc, s_r$ by a sequence $s^{\prime}_1, \dotsc, s^{\prime}_r$ so that
\[
s_1 k^{i_1} + s_2 k^{i_2} + \dotsc + s_r k^{i_r} \equiv s^{\prime}_1 k^{i_1} + s^{\prime}_2 k^{i_2} + \dotsc + s^{\prime}_r k^{i_r} \pmod{n},
\]
where $|s^{\prime}_1| + \dotsc + |s^{\prime}_r| \leq \wt(n, k; {\underline i}) <$ acs$(w({\underline b^{\prime}}, {\underline c}))$. Thus, replacing the terms $s_q$ by $s^{\prime}_q$ and then reconstructing the sequence ${\underline b^{\prime}}$ yields (iv).
\end{proof}

At this point we need to clarify the requirement of introducing the concept of formal sum. Let $g = x^a y^b \in G_{m,n,k}$ be connected to $1$ by a fixed reduced path
\[
P : x^{a_1} y^{b_1} x^{a_2} y^{b_2} \dotsc x^{a_t} y^{b_t} = x^a y^b.
\] 
Set $c_i := a_{i+1} + \dotsc + a_t$, for $1 \leq i \leq t-1$, and $c_t = 0$. The first reduction step essentially reduces the length of the path $l(P) = \sum_{i=1}^t |a_i|+|b_i|$ without changing $\sum_{i=1}^t |a_i|$. To do this we need to keep the recursive sequence $c_i$ intact, and this was the main outline of the proof above. Also, note that while $r = \alpha$ in the first reduction step, we have reduced to the case $\acs(w({\underline b^{\prime}}, {\underline c})) \leq \wt(n, k; \alpha)$ since there is a unique (ordered) sequence of length $\alpha$ in terms of the powers of $k$.  

 Now recall that the number $t$ is called the syllable of $P$ (which we denoted by $\syl(P)$ in Section~\ref{sec:intro}). Also, we denote the collection of all reduced paths from $1$ to $g \in G_{m,n,k}$ by $\mathcal{P}_g$. 

\begin{proposition}[Second reduction step] 
\label{prop:second_red_step}
Let  $P : \prod_{i=1}^t x^{a_i}y^{b_i}$ be a reduced path from $1$ to an element $g = x^a y^b \in G_{m,n,k}$. Suppose that $r$ is the number of distinct terms in the formal element 
$w({\underline b}, {\underline c}) = \prod_{i=1}^t b_i k^{c_i}$.
\begin{enumerate}[(i)]
\item If $w({\underline b}, {\underline c})$ is not primal as a formal element, then there is another $P' \in \mathcal{P}_g$ such that $l(P') \leq l(P)$. 
\item If $\acs(w({\underline b}, {\underline c})) > \wt(n, k; {\underline i})$ (where ${\underline i}$ is as denoted in Proposition {\ref{prop:first_red_step}}), then $P$ cannot be a shortest path in $\mathcal{P}_g$.
\item If $P \in \mathcal{P}_g$ is a path of shortest length, then $\syl(P) \leq \ord(k).$
\end{enumerate}
\end{proposition}

\begin{proof}
Parts (i)-(ii) follow directly from Proposition {\ref{prop:first_red_step}}. For (iii), note that the length of the ordered sequence $S(w({\underline b}, {\underline c}))$ is $\leq \alpha$. Using Proposition {\ref{prop:first_red_step}}, we reduce the path $P$ to $P^{\prime}$ and combine the terms that are powers of $x$ (which does not change the length of the path $P^{\prime}$).   
\end{proof}

\noindent Using the second reduction we may assume that $|b_1| + \dotsc + |b_t| \leq \wt(n, k; \alpha)$, which finally brings us to the main result in this paper.

\begin{theorem}[Main theorem] 
\label{thm:main}
Let $G_{m,n,k}$ be the split metacyclic group given by the presentation
\[
G_{m,n,k} = \langle x, y : x^m = 1 = y^{n}, x^{-1} yx = y^k \rangle,
\]
where $k$ has order $\alpha$ in the group ${\mathbb Z}^{\times}_n$ of units. If $\alpha$ is even and $k^{\alpha/2} \equiv -1 \pmod{n}$, then
\[
\diam(G_{m,n,k}) \leq 
\begin{cases}
[m/2] + \wt(n, k; \alpha), &\mbox{if } \alpha \neq m, \\
[m/2] + \wt(n, k; \alpha) + \deg(n,k;\alpha), &\mbox{if } \alpha = m.
\end{cases}
\]
\end{theorem}

\begin{proof}
We wish to bound the length of a path from $1$ to an element $g = x^a y^b \in G$. Set $i_1 = \deg(n,k;\alpha)$, and assume without loss of generality that $-[m/2] \leq a \leq [m/2]$. We break our argument into three cases. 

\begin{case} Assume that $i_1 \leq a \leq [m/2]$. Let ${\underline i}$ denote a minimal prime sequence with degree $i_1 = \deg(n,k;\alpha)$ given by
\[
\alpha - 1 \geq i_1 > i_2 > \dotsc > i_{t-1} > i_t = 0.
\]
First, we express $b \in {\mathbb Z}_n$ as 
\begin{equation}
\label{eqn:exp_for_b}
b \equiv b_1 k^{i_1} + b_2 k^{i_2} + \dotsc + b_{t-1} k^{i_{t-1}} + b_t \pmod{n}. 
\end{equation}
with $\sum_{j=1}^{t} |b_t| \leq \wt(n, k; \alpha)$. Take $\xi = a - i_1 \geq 0$, and consider the path
\begin{equation}
\label{eqn:exp_for_P}
P : x^{\xi} y^{b_1} x^{i_1 - i_2} y^{b_2} x^{i_2 - i_3} y^{b_3} \dotsc x^{i_{t-2} - i_{t-1}} y^{b_{t-1}} x^{i_{t-1}} y^{b_t}.
\end{equation}
Clearly, $P \in \mathcal{P}_g$, and since every exponent of $x$ in $P$ is non-negative, we have
\[
l(P) = (a - i_1) + (i_1 - i_2) + \dotsc + (i_{t-2} - i_{t-1}) + i_{t-1} + \sum_{j=1}^{t} |b_j| \leq [m/2] + wt(n, k; \alpha),
\]
which proves the result for this case. 
\end{case}

\begin{case} Assume that $-[m/2] \leq a \leq -i_1$. Note that for any path $P \in \mathcal{P}_g$, we have $P^{-1} \in \mathcal{P}_{g^{-1}}$ and $l(P) = l(P^{-1})$, where $P^{-1} = \prod_{i=0}^{t-1} y^{-b_{t-i}}x^{-a_{t-i}}$. Consider $b^{\prime} \in {\mathbb Z}_n$ such that $y^{-b} x^{-a} = x^{-a} y^{b^{\prime}}$. Since the exponent $-a$ of $x$ satisfies the hypothesis of Case 1, the result for this case follows.
\end{case}

\begin{case} Finally, assume that $-i_1 < a < i_1$. As in previous case, it suffices to assume that $0 \leq a < i_1$. Write $b$ as in Equation~\ref{eqn:exp_for_b} above, set $\xi = a - i_1$, and consider a path $P'$ of the form in Equation~\ref{eqn:exp_for_P}. Clearly, $P' \in \mathcal{P}_g$, and further note that every other exponent, except the first exponent of $x$ in $P'$ is non-negative. Hence, we have 
\[
l(P) = -a + i_1 + (i_1 - i_2) + (i_2 - i_3) + \dotsc + (i_{t-2} - i_{t-1}) + i_{t-1} = -a + 2i_1.
\]
Applying Proposition~\ref{thm:bound_for_deg_nkalpha}, we get
$$l(P) \leq [m/2] + i_1.$$ The result now follows from the fact that 
$$\alpha \leq \begin{cases}
[m/2], & \text{if } \alpha \text{ is a proper divisor of $m$, and} \\
[m/2] + i_1, & \text{otherwise.}
\end{cases}$$
\end{case}
\end{proof}

\noindent Note that the third case of Theorem~\ref{thm:main} used the fact that $2 i_1 \leq \alpha$. However, when $n$ is a prime $p \equiv 1 \pmod{4}$, we know that $2 i_1 \leq [\alpha / 2]$, which leads to a better bound. More generally, we have the following:

\begin{corollary} 
\label{cor:diam_bound_for_primes}
Let $G_{m,n,k}$ be the split metacyclic group given by the presentation
\[
G_{m,n,k} = \langle x, y : x^m = 1 = y^{n}, x^{-1} yx = y^k \rangle,
\]
where $k$ has order $\alpha$ in the group ${\mathbb Z}^{\times}_n$ of units. If $\alpha$ is even with $k^{\alpha/2} \equiv -1 \pmod{n}$ and $2\deg(n,k;\alpha) \leq [\alpha/2]$, then
$$\diam(G_{m,n,k}) \leq [m/2] + \wt(n,k;\alpha).$$
\end{corollary}

\noindent The fact that every metacyclic group $G$ is a quotient of a split metacyclic group yields a bound for $\diam(G)$.

\begin{corollary}
Let $G_{m_0,\ell,n,k}$ be an arbitrary metacyclic group given by the presentation
$$G_{m_0,\ell,n,k} = \langle x, y : x^{m_0} = y^{\ell}, \, y^{n} = 1, \, x^{-1} yx = y^k \rangle,$$
where $k^{m_0} \equiv 1 \pmod{n}$, $n \mid \ell(k-1)$ and $\ell \mid n$. Then
$$\diam(G_{m_0,\ell,n,k}) \leq  \diam(G_{m_0n/\ell,n,k}) \leq \left[\frac{m_0n/\ell}{2}\right] + \wt(m_0n/\ell,k;m_0).$$ 
\end{corollary}
\begin{proof}
First, we note that the condition $\ell \mid n$ does not violate the generality of the above presentation (see~\cite[Lemma 2.1]{H1}). Clearly, there exists a natural surjection $G_{m_0n/\ell,n,k} \twoheadrightarrow G_{m_0,\ell,n,k}$, and the result follows.
\end{proof}

\section{Some explicit computations}
\label{sec:compute_weights}
When $n$ is a prime $p \equiv 1 \pmod{4}$ (resp. $\equiv 3 \pmod{4}$), we showed in Theorem~\ref{thm:prime_wt_bounds} that $\wt(p, k; p-1) \leq \frac {p+3}{4}$ (resp. $\leq \frac {p+5}{4}$). Nevertheless, in practice (as we will see), the value of $\wt(p, k; p-1)$ is much less than these bounds. 

In Tables~\ref{tab:primes_1mod4} and~\ref{tab:primes_3mod4} below, we list several computations of $\wt(n, k; \alpha)$ for various primes $n$ and primitive units $k \in \mathbb{Z}_n ^{\times}$. Further, for these values of $n$, we consider $m = n-1$, and indicate how the values of the bound (derived in Corollary~\ref{cor:diam_bound_for_primes}) compares with the actual values of $\diam(G_{m,n,k})$. These computations were made using software written in Mathematica 11~\cite{W1}.  

\begin{table}[!htbp]
\begin{center}
\small\addtolength{\tabcolsep}{-4pt}
  \begin{tabular}{ | l | c | c | c | c | }
    \hline
    $(m, n, k)$ &  $\wt(n,k;\alpha)$ & ${\underline \lambda}$ realizing $\wt(n,k;\alpha)$ & $\diam(G_{m,n,k})$ & $[m/2] + \wt(n,k;\alpha)$ \\ \hline
    $(12, 13, 2)$ & $3$ & $1,1,1$ &$7$ & $9$ \\ \hline
    $(16, 17, 3)$ & $3$ & $0, 1, 1, 1$ & $9$ & $11$  \\ \hline
    $(28, 29, 2)$ & $4$ & $0, 0, 0, 1, 1, 1, 1$ & $15$ & $18$ \\ \hline
    $(36, 37, 2)$ & $4$ & $0, 0, 0, 1, 0, 0, 1, 1, 1$ & $19$ & $22$ \\ \hline
    $(40, 41, 6)$ & $4$ & $0, 0, 0, 0, 1, 0, 1, 0, 1, 1$ & $21$ & $24$ \\ \hline
    $(52, 53, 2)$ & $4$ & $0, 0, 0, 1, 1, 0, 0, 0, 0, 1, 0, 0, 1$ & $27$ & 30 \\ \hline
    $(60, 61, 2)$ & $4$ & $0, 0, 0, 0, 0, 1, 0, 0, 1, 0, 0, 1, 0, 0, 1$ & $31$ & $34$ \\ \hline 
  \end{tabular}
  \caption{Values of $\wt(p,k;p-1)$, for some primes $p \equiv 1 \pmod{4}$.}
  \label{tab:primes_1mod4}
\end{center}
\end{table}

\begin{table}[!htbp]
\begin{center}
\small\addtolength{\tabcolsep}{-4pt}
  \begin{tabular}{ | l | c | c | c | c | }
    \hline
    $(m, n, k)$ & $\wt(n,k;\alpha)$ & ${\underline \lambda}$ realizing $\wt(n,k;\alpha)$ & $\diam(G_{m,n,k})$ & $[m/2] + \wt(n,k;\alpha)$ \\ \hline
    $(6, 7, 3)$ & $2$ & $0,1,1$ & $4$ & $5$  \\ \hline
    $(10, 11, 2)$ & $3$ & $0,0,1,2$ & $6$ & $8$\\ \hline
    $(18, 19, 2)$ & $3$ & $0, 1, 0, 0, 1, 1$ & $10$ & $12$\\ \hline
    $(22, 23, 5)$ & $3$ & $1, 0, 0, 0, 0, 1, 1$ & $12$ & $14$\\ \hline
    $(30, 31, 3)$ & $4$ & $0, 0, 0, 0, 0, 0, 1, 1, 2$ & $16$ & $19$\\ \hline
    $(42, 43, 3)$ & $4$ & $0, 0, 0, 0, 0, 0, 0, 0, 0, 1, 1, 2$ & $22$ & $25$\\ \hline
    $(46, 47, 5)$ & $4$ & $0, 0, 0, 0, 0, 0, 0, 0, 0, 1, 1, 1, 1$ & $24$ & $27$ \\ \hline
    $(58, 59, 2)$ & $4$ & $0, 0, 0, 0, 1, 0, 0, 0, 0, 1, 0, 0, 1, 0, 0, 1$ & $30$ & $33$ \\ \hline 
  \end{tabular}
     \caption{Values of $\wt(p,k;p-1)$, for some primes $p \equiv 3 \pmod{4}$.}
      \label{tab:primes_3mod4}
\end{center}
\end{table}
 
\bibliographystyle{plain} 
\bibliography{Rev-semidirect}

\begin{thebibliography}{10}

\bibitem{B1}
L.~Babai, G.~Hetyei, W.~M. Kantor, A.~Lubotzky, and \'A. Seress.
\newblock On the diameter of finite groups.
\newblock In {\em 31st {A}nnual {S}ymposium on {F}oundations of {C}omputer
  {S}cience, {V}ol.\ {I}, {II} ({S}t.\ {L}ouis, {MO}, 1990)}, pages 857--865.
  IEEE Comput. Soc. Press, Los Alamitos, CA, 1990.

\bibitem{C1}
Fan~RK Chung.
\newblock Diameters and eigenvalues.
\newblock {\em Journal of the American Mathematical Society}, 2(2):187--196,
  1989.

\bibitem{D4}
Richard~N Draper.
\newblock {\em A fast distributed routing algorithm for supertoroidal
  networks}.
\newblock Supercomputing Research Center, 1990.

\bibitem{D3}
Richard~N Draper.
\newblock An overview of supertoroidal networks.
\newblock In {\em Proceedings of the third annual ACM symposium on Parallel
  algorithms and architectures}, pages 95--102. ACM, 1991.

\bibitem{D1}
Richard~N Draper and Vance Faber.
\newblock {\em The diameter and mean diameter of supertoroidal networks}.
\newblock Supercomputing Research Center, 1990.

\bibitem{D2}
Richard~N Draper and Vance Faber.
\newblock The diameter and average distance of supertoroidal networks.
\newblock {\em Journal of Parallel and Distributed Computing}, 31(1):1--13,
  1995.

\bibitem{E1}
Grahame Erskine.
\newblock Diameter 2 {C}ayley graphs of dihedral groups.
\newblock {\em Discrete Math.}, 338(6):1022--1024, 2015.

\bibitem{H1}
C.~E. Hempel.
\newblock Metacyclic groups.
\newblock {\em Comm. Algebra}, 28(8):3865--3897, 2000.

\bibitem{MH1}
Heather Macbeth, Jana \v~Siagiov\'a, and Jozef \v~Sir\'a\v~n.
\newblock Cayley graphs of given degree and diameter for cyclic, {A}belian, and
  metacyclic groups.
\newblock {\em Discrete Math.}, 312(1):94--99, 2012.

\bibitem{NAT}
Melvyn~B. Nathanson.
\newblock {\em Additive number theory}, volume 165 of {\em Graduate Texts in
  Mathematics}.
\newblock Springer-Verlag, New York, 1996.
\newblock Inverse problems and the geometry of sumsets.

\bibitem{W1}
Wolfram Research.
\newblock Mathematica 11.0, 2016.

\bibitem{VT1}
Tom\'a\v~s Vetr\'\i~k.
\newblock Abelian {C}ayley graphs of given degree and diameter 2 and 3.
\newblock {\em Graphs Combin.}, 30(6):1587--1591, 2014.

\bibitem{LV1}
Fen~Lin Wu, S~Lakshmivarahan, and Sudarshan~K. Dhall.
\newblock Routing in a class of cayley graphs of semidirect products of finite
  groups.
\newblock {\em Journal of Parallel and distributed computing}, 60(5):539--565,
  2000.

\end{thebibliography}
\end{document}